\documentclass[12pt,leqno,twoside,a4paper]{article}
\usepackage{amsfonts,amsmath,amsthm,amssymb}
\usepackage[english]{babel}
\usepackage{color}
\newtheorem{theorem}{Theorem}[section]
\newtheorem{proposition}[theorem]{Proposition}
\newtheorem{lemma}[theorem]{Lemma}
\newtheorem{corollary}[theorem]{Corollary}
\theoremstyle{definition}
\newtheorem{definition}[theorem]{Definition}
\newtheorem{definitions}[theorem]{Definitions}

\newtheorem{example}[theorem]{Example}
\newtheorem{examples}[theorem]{Examples}
\theoremstyle{remark}

\newtheorem{remarks}[theorem]{\bf Remarks}

 \numberwithin{equation}{section}

\newcommand{\si}{\sigma}
\newcommand{\Ore}{A[t;\sigma,\delta]}
\newcommand{\de}{\delta}

\theoremstyle{definition}
\theoremstyle{remark}
\numberwithin{equation}{section}

\begin{document}
\title{($\sigma,\delta$)-codes}
\date{}
\author{{\bf \normalsize M'Hammed Boulagouaz and Andr\'{e} Leroy}\footnote{}
\vspace{6pt}\\
\normalsize University of King Khalid,\\
Abha, Saudi Arabia\\
    \normalsize E-mail: boulag@yahoo.com \vspace{6pt}\\
\normalsize  Universit\'{e} d'Artois,  Facult\'{e} Jean Perrin\\
\normalsize Rue Jean Souvraz  62 307 Lens, France\\
 \normalsize  E-mail: leroy@euler.univ-artois.fr }
\maketitle\markboth{\rm  $(\sigma,\delta)$-codes by M. Boulagouaz
and A. Leroy}{ \rm $(\sigma,\delta)$-codes by M. Boulagouaz and A.
Leroy}
\maketitle\markboth{\rm  M. Boulagouaz and A. Leroy}{ \rm ($\sigma,\delta$)-codes}

\textbf{Abstract}:  In this paper we introduce the notion of
cyclic ($f(t),\sigma,\delta$)-codes for $f(t)\in \Ore$.  These codes
generalize the $\theta$-codes as
introduced by D. Boucher, F. Ulmer, W. Geiselmann \cite{BGU}.   We construct 
generic and control matrices for these codes.  As a particular case 
the ($\si,\de$)-$W$-code associated to a Wedderburn 
polynomial are defined and we show that their control matrices are given by 
generalized Vandermonde matrices.  All the Wedderburn polynomials of 
$\mathbb F_q[t;\theta]$ are described and their control matrices are presented.  
A key role will 
be played by the pseudo-linear transformations.

\section{Introduction and preliminaries}
\vspace{4mm}

The use of rings in coding theory started when it appears that working over rings 
allowed certain codes to be looked upon as linear codes.  The use of {\it 
noncommuative} rings emerged recently in coding theory due to the pertinence of 
Frobenius rings for generalizing
Mac Williams theorems (cf. \cite{W}, for details) and also because
of the use of Ore polynomial rings as source of generalizations of cyclic codes 
(cf. e.g. \cite{BGU},\cite{BSU},\cite{SY}).
With some few exceptions (e.g. \cite{LS},\cite{BU}) the Ore polynomial rings used so far 
in coding theory are mainly of automorphisms type with a (finite) field as base ring. This paper shows
how one can use general Ore extensions to not only define codes, but as well give 
their generic and control matrices.  Factorizations techniques in Ore polynomial 
rings play an important role in these questions and the interested reader can 
consult \cite{G}, \cite{LL} or \cite{LLO} for more information on this matter. 
Since they are intimately related to modules over Ore 
extensions and to factorizations, the pseudo-linear transformations will play an 
important role in this paper.
The reader may consult \cite{J},\cite{L} and \cite{L2} for more details on 
pseudo-linear transformations.
\begin{definitions}
Let $A$ be a ring with $1$ and $\sigma$ a ring endomorphism of $A$.
\begin{enumerate}
\item[(a)] An additive map $\delta \in End(A,+)$ is a $\sigma$-derivation if, for any $a,b\in A$, we have:
$$
 \de(ab)=\si(a)\de(b)+\de(a)b.
$$
\item[(b)] Let $\delta$ be a $\sigma$-derivation of a ring $A$.
The elements of the skew polynomial ring $R=A[t;\si,\de]$ are sums $\sum a_it^i$.
 They are added as ordinary polynomials and the multiplication is based on the
commutation law $$ta=\si(a)t+\de(a), \; for \; a\in A.$$
\item[(c)] The degree of a nonzero polynomial $f = a_0 + a_1t + a_2t^2+\dots a_nt^n\in R=A[t;\si,\de]$
is defined to be $\deg(f) = max\{i\vert a_i\ne 0\}$ and we put, as usual,
$\deg(0)=-\infty$.
\end{enumerate}
\end{definitions}

\begin{examples}
\label{examples of Skew polynomial rings}
{\rm
\begin{enumerate}
\item[(1)] If $\si=id.$ and $\de=0$ we have $A[t;\si,\de]=A[t]$, the usual polynomial
ring in a commuting variable.
If only $\si=id.$ but $\de\ne 0$ we denote $A[t;id.,\de]$ as
$A[t;\de]$ and speak of a polynomial ring of derivation type.
On the other hand, if $\de=0$ but $\si\ne id.$,
we write $A[t;\si,\de]$ as $A[t;\si]$ and refer to this Ore extension as a polynomial ring of endomorphism type.
\item[(2)] Let $\si$ stand for the usual conjugation of the complex
number $\mathbb C$ and consider $\mathbb C[t;\si,0]$.  Notice that, since $\si^2=id.$, we can check that $t^2$ is a central polynomial.
\item[(3)] Let $k$ be field, $R=k[x][t;id.;d/dx]$.  This is the weyl algebra.  The commutation law is $tx-xt=1$.  If char$k=0$ the Weyl algebra is a simple ring.  In contrast if
char$k=p>0$ then $t^p$ and $x^p$ are central elements.
\item[(4)] For $a\in A$, we define the inner $\si$-derivation induced by $a$ (denoted $d_{a,\si}$) in the following way: for $r\in
A$, $d_{a,\si}(r):=ar-\si(r)a$.  Let us remark that
$A[t;\si,d_{a,\si}]=A[t-a,\si]$.   Similarly, for an inner
automorphism $I_a$ induced by an invertible element $a\in A$ and
defined by $I_a(x)=axa^{-1}$ for $x\in A$, we have
$A[t;I_a]=A[a^{-1}t]$.  Let us mention that an easy computation
shows that if there exists a central element $c\in Z(A)$, where
$Z(A)$ denotes the center of $A$, such that $c-\sigma(c)$ is an
invertible element of $Z(A)$, then the derivation $\delta$ is inner
induced by $(c-\sigma(c))^{-1}\delta (c)$.  In particular, if $A$ is
a field then either $\sigma=id$ or $\delta$ is inner.  More
particularly, all Ore extensions built on a finite field $\mathbb
F_q$ are of the form $\mathbb F_q[t;\theta]$ where $\theta$ is an
automorphism of $\mathbb F_q$.
\item[(5)] It is well-known that finiteness conditions force 
($\sigma$)-derivations to be 
inner (Cf. e.g. \cite{A}).  
  We now give an easy example of a (finite) ring having a non-inner $\sigma$-derivation.  Let $K$ be any
 ring and $\sigma$ a non inner automorphism of $K$.  Consider the ring $A\subset M_2(K)$  defined by:
 $$
 A:=\left\lbrace \begin{pmatrix}
a & b \\
0 & a
\end{pmatrix} \; \vert \; a,b\in K,\;\sigma(a)=a \right\rbrace .
 $$
 We extend $\sigma$ to $A$ by letting it act on each coefficient of the matrices and define the
 additive map $\delta$ by setting:
$$
\delta \left(\begin{pmatrix}
a & b \\
0 & a
\end{pmatrix}\right)=\begin{pmatrix}
0 & \sigma(b) \\
0 & 0
\end{pmatrix}\quad {\rm for}\quad a,b \in K\quad {\rm with}\;\sigma(a)=a.
$$
One can check that this map is indeed a non inner sigma-derivation.  For instance
one can put $K=\mathbb{F}_q$ and let $\sigma$ be the Frobenius map.  Hence, in
this case, $A$ is a finite ring with a non-inner
$\sigma$-derivation.

\item[(6)] Let $p$ be a prime number, $n\in \mathbb N$ and $q=p^n$.  Consider
$R=\mathbb F_q$ the finite field with $q$ elements and $\theta$ the
Frobenius automorphism defined by $\theta(a)=a^p$ for $a\in \mathbb F_q$. Skew
polynomial ring $\mathbb{F}_q[t;\theta]$ have been used recently in
the context of noncommutative codes.  The main advantage of
$R=\mathbb{F}_q[t;\theta]$ versus the classical $\mathbb{F}_q[x]$ is
that a given polynomial $p(t)\in R$ admits generally many different
factorizations. For instance let us consider $\mathbb F_4=\mathbb
F_2[\alpha]$, where $\alpha^2+\alpha +1=0$, and some factorizations
of $t^4+1\in \mathbb{F}_4[t;\theta]$:
\[\begin{aligned}
t^4+1 &=(t^2+1)(t^2+1)\\
      &=(t^2+\alpha t +\alpha)(t^2+\alpha t +\alpha^2)\\
      &=(t^2+\alpha^2 t +\alpha^2)(t^2+\alpha^2 t +\alpha)\\
      &=(t^2+\alpha t +\alpha^2)(t^2+\alpha t +\alpha)=...
\end{aligned}\]
\end{enumerate}
}
In fact, it has been shown recently that factorizations in $\mathbb F_q[t;\theta]$ can be
worked out from factorizations in $\mathbb F_4[x]$ (cf. \cite{L2}).
\end{examples}

\section{Polynomial and pseudo-linear maps }

Let $A,\sigma$ and $\delta$ be a ring, an endomorphism and a
$\sigma$-derivation of $A$, respectively.  Let us put
$R=A[t;\sigma,\delta]$.\\

For any $f(t)\in R$ and $a\in A$ there
exists a unique $q(t)\in A[t,\sigma,\delta]$ and   a
unique $s\in A$ such that:
$$f(t)=q(t)(t-a) +
s.$$
\begin{definitions}
\label{definitions of polynomial maps and pseudo linear maps}
\begin{enumerate}
\item[(a)] With these notations, the (right) polynomial map
associated to $f(t)\in R$ is
$$f:A\longrightarrow A \quad {\rm given \; by}\quad
 f(a):=s$$
\item[(b)] For $i\ge 0$, the
right polynomial map determined by $t^i$ will be denoted by $N_i$.   With these notations one has that
$(\sum_{i=0}^nb_it^i)(a)=\sum_{i=0}^nb_iN_i(a)$ for any polynomial 
$f(t)=\sum_{i=0}^nb_it^i\in R$.
When $\delta=0$ one has $N_i(a)=\sigma^{i-1}(a)\sigma^{i-2}(a)\dots\sigma(a)a$, this justifies the notation $N_i$.
\item[(c)]  Let $_AV$ be a left $A$-module.  An additive map
$T:V\longrightarrow V$ such that,
for $\alpha \in A$ and $v\in V$,
$$
T(\alpha v)=\sigma(\alpha)T(v) + \delta(\alpha)v.
$$
is called a ($\si,\de$) pseudo-linear transformation
(or a ($\si,\de$)-PLT, for short).
\end{enumerate}
\end{definitions}

\begin{examples}
\label{examples of PLT}
\begin{enumerate}
\item If $\si=id.$ and $\de=0$ we get back the standard way of
evaluating a polynomial.  It should be noted though, that, since
$R$ is not commutative, we have to specify that this is a right
polynomial map.  For instance, although for $c\in A$
$f(t)=ct=tc\in R=A[t]$, the polynomial map we consider here
is the map $f:A\longrightarrow A$ defined by $f(a)=ca$, for any
$a\in A$.
\item Let $A,\sigma$ and $\delta$  be a ring, an endomorphism of $A$ and a
$\sigma$-derivation of $A$ respectively.
If $a\in A$, the map
$$
T_a:A\longrightarrow A
\quad x \mapsto T_a(x)=\sigma(x)a+\delta(x)
$$
is a ($\sigma,\delta$)-PLT defined on the left $A$-module: $_AA$.

$\diamond$ if $\sigma=id$ and $\delta=0$, we get $T_a(x)=xa$.\\
$\diamond$ if $a=0$, we get $T_0=\delta$\\
$\diamond$ if $a=1$ and $\delta=0$, we get that
$T_1=\sigma$.\\
$\diamond$ if $a=1$, we get that
$T_1=\sigma+\delta$.\\

\item Let $V$ be a free left $A$-module with basis $\beta=\{e_1,\dots ,e_n\}$ and let $T:V \rightarrow V$ be a ($\si, \de$)-PLT.
This gives rise to a ($\si,\de$)-PLT on the left $A$-module $A^n$ as follows:
first define $C=(c_{ij})\in M_n(A)$ by
$
T(e_i)=\sum_i^nc_{ij}e_j.
$
  Then we extend component-wise $\si$ and $\de$ to the ring
$A^n$.  Finally we then define a ($\si,\de$)-PLT on $A^n$
(considered as a left $A$-module) by
$T_C(\underline{v})=\sigma(\underline{v})C+\delta(\underline{v})$, for $\underline{v}\in
A^n$. Indeed, it is easy to check that we have
$T_C(\alpha \underline{v})=\sigma(\alpha) T_C(\underline{v}) + \delta(\alpha)\underline{v} $.
\item[(4)] Let us remark that powers of a ($\sigma,\delta$)-pseudo-linear transformation defined on a left $A$-module $V$ are usually not pseudo-linear.  
For instance it is easy to check that, for $\alpha\in A$ and $v\in V$, we have
$$
T^n(\alpha v)=\sum_{i=0}^nf^n_i(\alpha)T^i(v),
$$
where $f^n_i$ stands for the sum of all words in $\sigma$ and $\delta$ having 
$n-i$ letters $\delta$ and $i$ letters $\sigma$.
\end{enumerate}
\end{examples}

In order to increase the sources of Codes, ring structures (and two sided ideals) have been replaced by modules
(and one-sided ideals).  In the case of modules over Ore extensions the next proposition shows that pseudo-linear maps are unavoidable.  In fact, they are also very useful since they are intimately related to
factorizations.  For instance they offer a generalization of the classical fact that in a commutative setting the evaluation map is a ring homomorphism (cf.  Lemma \ref{pq(T)=p(T)q(T)}).
\begin{proposition}
\label{PLT and R modules} Let $A$ be a ring $\si\in End(A)$ and
$\de$ a $\si$-derivation of $A$.   For an additive group $(V,+)$ the
following conditions are equivalent:
\begin{enumerate}
\item[(i)] $V$ is a  left $R=A[t;\si, \de]$-module;
\item[(ii)] $V$ is a left $A$-module and there exists a $(\si,\de)$ pseudo-linear
transformation $T:V\longrightarrow V$;
\item[(iii)] There exists a ring homomorphism $\Lambda: R\longrightarrow
End(V,+)$.
\end{enumerate}
\end{proposition}
\begin{proof}
(i)$\Longrightarrow$ (ii) The pseudo-linear map is given by the
left multiplication by $t$.

(ii)$\Longrightarrow$ (iii)  The ring homomorphism $\Lambda:R \longrightarrow End(V,+)$ is
defined by $\Lambda (f(t))=f(T)$, where, for $f(t)=\sum_{i=0}^na_it^i\in R$,  $f(T)$
stands for $\sum_{i=0}^na_iT^i\in End(V,+)$.

(iii)$\Longrightarrow$(i) This is classical.
\end{proof}

\vspace{4mm}

As a special case of the example (3) above let us mention the important pseudo-
linear transformation associated to
a given monic polynomial $f$ of degree $n$
(or rather to its companion matrix $C_f$).   The left $A$-module $V$ is, in this
case, $R/Rf$.  This is given in the following definition.
\begin{definition} Let $f(t)=\sum_{i=0}^na_it^i\in A[t;\sigma,\delta]$ be a monic 
polynomial of degree
$n$ and let 
$$C_f=\begin{pmatrix}
0 & 1 & 0 & \dots & 0 \\ 
0 & 0 & 1 & 0 & \dots \\ 
\vdots & \vdots & \vdots & \vdots & \vdots \\ 
0  & 0 & 0 & 0 & 1 \\ 
-a_0 & -a_1 & \dots & \dots & -a_{n-1}
\end{pmatrix} $$
be its companion matrix. Then the map
$$
T_f:A^n\longrightarrow A^n\; {\rm defined\; by} \;
T_f(x_1,...,x_n):=(\sigma(x_1),...,\sigma(x_n))C_f
+(\delta(x_1),...,\delta(x_n))
$$
is a pseudo-linear transformation called the pseudo-linear transformation
associated to  $f$.
\end{definition}

\begin{example}
For $a\in A$ the map:\\
\hspace*{3cm} $T_a:A\longrightarrow A$\\
\hspace*{4cm}$ x \longrightarrow T_a(x)=\sigma(x)a+\delta(x)$\\
is the  pseudo-linear transformation associated to $f(t)=t-a$.
\end{example}

\begin{proposition}
\label{p(T_a)(1)=p(a)}
 Let $a$ be an element in $A$ and $p(t)\in A[t;\sigma,\delta]$.  Then:
\begin{enumerate}
\item[(1)] $N_0(a)=1$ and for $i\ge 0$, $N_{i+1}(a)=\sigma(N_i(a))a+\delta(N_i(a))$.
\item[(2)] $p(T_a)(1)=p(a)$.
\end{enumerate}
\end{proposition}
\begin{proof}
(1) Since $N_i(a)$ is the remainder upon right division of $t^i$ by
$t-a$, we have $t^{i+1}=tt^i=t(q_i(t)(t-a)+N_i(a))=tq_i(t)(t-
a)+\sigma(N_i(a))t+\delta(N_i(a))=( tq_i(t)+\sigma(N_i(a)) )(t-
a)+\sigma(N_i(a))a+\delta(N_i(a))$.  This gives the required
equality.

\noindent (2) It is enough to show that, for $i\ge 0$,
$N_i(a)=T_a^i(1)$. This is clear for $i=0$.  Using the point (a)
above and an induction we get $N_{i+1}(a)=\sigma(N_i(a))a+
\delta(N_i(a))=T_a(N_i(a))=T_a(T_a^i(1))=T_a^{i+1}(1)$.
\end{proof}

\begin{lemma}
\label{pq(T)=p(T)q(T)}
Let $T:V \longrightarrow V$ be a ($\sigma, \delta$) pseudo-linear transformation defined on a left $A$-module $V$.  Then for any polynomial $p(t),q(t)\in R=
A[t;\sigma,\delta]$ we have $(p(t)q(t))(T)=p(T)q(T)\in End(V,+)$.
\end{lemma}
\begin{proof}
Let us recall that if $p(t)=\sum_ia_it^i\in R$ then $p(T)$ is the additive
endomorphism defined by $p(T)(v)=\sum_ia_iT^i(v)$, for $v\in V$.
Since for $v\in V$ the left
$R$-module structure of $V$ is induced by $t.v:=T(v)$, the equality given in the
 lemma is in fact a simple translation of the fact that $(p(t)q(t)).v=p(t).
 (q(t).v)$ for $v\in V, \;p(t), q(t)\in R$.
\end{proof}

\vspace{3mm}

The formula $p(T_a)(1)=p(a)$ in Proposition \ref{p(T_a)(1)=p(a)} can be interpreted
as saying that $p(t)\in R(t-a)$ if and only if $p(T_a)(1)=0$.
This can be generalized as follows: for a monic polynomial 
$f(t)\in R=A[t;\sigma,\delta]$
we denote, as earlier, $T_f$ the pseudo-linear map defined on $A^n$ by the companion matrix of
$f$.  We then have $p(t)\in Rf(t)$ if and only if $p(T_f)(1,0,\dots,0)=(0,\dots,0)$
(cf. Theorem 1.10 in \cite{L2}).
Making use of the above lemma \ref{pq(T)=p(T)q(T)} we then easily get
that, for $p(t),q(t)\in R$ with $\deg(q)<\deg(f)$, $p(t)q(t)\in Rf$ if and
only if $p(T_f)q(T_f)(1,0,\dots,0)=(0,\dots,0)=p(T_f)(\underline{q})$.
We refer the reader
to \cite{L2} for details.  For easy reference, let us sum up this in the following lemma.

\vspace{3mm}

\begin{lemma}
\label{f divise à droite pq}
Let $f(t),p(t),q(t)$ be polynomials in $R=\Ore$ such that $f(t)$ is monic 
and $\deg(q)<\deg(f)=n$ then
$p(t)q(t)\in Rf(t)$ if and only if $p(T_f)(\underline{q})=(0,\dots,0)$, where, for 
$q(t)=\sum_{i=0}^{n-1} q_it^i$, we denote $\underline{q}$ the $n$-tuple 
$(q_0,q_1,\dots,q_{n-1})$.
\end{lemma}

For a monic polynomial $f(t)\in R=A[t;\sigma,\delta]$ of degree $n$ let us 
mention the following proposition which shows how to
translate results from the $R/Rf$ to $A^n$.

\begin{proposition}
\label{bijection between R/Rf and A^n}
Let $f(t) \in R=\Ore$ be a monic polynomial of degree $n>0$.
The map $\varphi : R/Rf(t)\longrightarrow  A^n$ given by $\varphi(p+Rf)=p(T_f)(1,0,
\dots,0)$ is a bijection.
\end{proposition}
\begin{proof}
Since $T_f$ represents the left multiplication by $t$ on $R/Rf(t)$ and since this
corresponds to the pseudo-linear transformation $T_f$, the above bijection is
clear.
\end{proof}

\vspace{3mm}

The above bijection endows $A^n$ with a left $R=A[t;\sigma,\delta]$-module structure.

Let us remark that if $(a_0,a_1,\dots,a_{n-1})\in A^n$ then
$\varphi(\sum_{i=0}^{n-1}a_it^i+Rf)=(a_0,\dots,a_{n-1})$.  Notice also that
the practical effect of this proposition is a way of computing the remainder of
the euclidean right division by $f(t)$.

\vspace{4mm}

\section{Generic and control matrices of ($\sigma,\delta$)-codes}

\vspace{4mm}

Let $A$ be a ring, $\sigma,\delta$ be an endomorphism and a $\sigma$-derivation
of $A$ respectively.

\begin{definitions}
Let $f(t)$ be a monic polynomial in $R=\Ore$.
A cyclic ($f, \si,\de$)-code
is the image $\varphi (Rg/Rf)$ of the cyclic module $Rg/Rf$ where $g(t)\in \Ore$ is a monic polynomial such that $f(t)\in Rg(t)$ and $\varphi$ is the map described in Proposition \ref{bijection between R/Rf and A^n}.
A cyclic ($f,\sigma,\delta$)-code $C\subseteq A^n$ is then the subset of $A^n$ consisting of the coordinates of the elements of $Rg/Rf$ in the basis
$\{1,t,\dots, t^{n-1}\}$ for some right monic factor $g(t)$ of $f(t)$.
\end{definitions}

In the next theorem we answer a few natural questions related to these notions.

\vspace{2mm}

\begin{theorem}
\label{generic matrices}
Let $g(t):=g_0+g_1t+\dots +g_{r}t^r\in R$ be a monic polynomial ($g_r=1$).  With the above notations we have
\begin{enumerate}
\item[(a)] The code corresponding to $Rg/Rf$ is a free left $A$-module
of dimension $n-r$ where $\deg(f)=n$ and $\deg(g)=r$.
\item[(b)] If $v:=(a_0,a_1,\dots,a_{n-1})\in C$ then $T_f(v)\in C$.
\item[(c)] The rows of the matrix generating the code $C$
are given by $$(T_f)^k(g_0,g_1,\dots,g_r,0,\dots,0),
\quad {\rm for}\;\, 0\le k \le n-r-1.$$ 
\end{enumerate}
\end{theorem}
\begin{proof}
\noindent (a) We have $f=hg$ for some monic polynomial $h\in R$.  Hence as left
$R$-modules we have also $Rg/Rf\cong R/Rh$.  Since $h$ is monic $R/Rh$ is a free
$A$-module of rank $\deg(h)=n-r$.

\noindent (b) $v=(a_0,\dots ,a_{n-1})\in C$ if and only if
$q(t):= \sum_{i=0}^na_it^i+ Rf\in Rg/Rf$.  Since $tq(t)\in Rg/Rf$ and left
multiplication by $t$ on $R/Rf$ corresponds to the action of $T_f$ on $A^n$, we do get that $T_f((a_0,\dots , a_{n-1}))\in C$, as required.

\noindent (c) Clearly for any $k\ge 0$ we have that $T_f^k(g_0,\dots,g_{r},0,\dots,0)\in
C$.  On the other hand it is clear that $g+Rf,tg+Rf,\dots, t^{n-r-1}g+Rf$ are
left linearly independent over $A$ and hence constitutes a basis of $Rg/Rf$.  In
terms
of code words this gives that the vectors $T_f^{k}((g_0,\dots , g_{r},0,\dots,0))$ for
$0\le k \le n-r-1$ form a left $A$-basis of the code $C$.
\end{proof}

\vspace{4mm}

\begin{examples}
\label{examples of generic matrices}
In the five first examples hereunder $A=\mathbb F_{p^n}$ stands for a finite
field.
\begin{enumerate}
\item[(1)] If $\sigma=Id.$, $\delta=0$, $f=t^n-1$ and $f=gh$\\
$(b)$ gives the cyclicity condition for the code.\\
 $(c)$ we get the standard generating matrix of a cyclic code.
\item[(2)] If $\sigma=Id.$, $\delta=0$, $f=t^n-\lambda$ and $f=gh$\\
$(b)$ gives the constacyclicity condition for the code.\\
 $(c)$ we get the standard generating matrix of a constacyclic code.
\item[(3)] $f=t^n-1\in R=\mathbb{F}_q[t;\theta]$ ($\theta="Frobenius"$) and $f=gh\in R$ \\
$(b)$ gives the $\theta$-cyclicity condition for the code.\\
 $(c)$ we get the standard generating matrix of a $\theta$-cyclic code (cf. \cite{BGU}).
\item[(4)] If $\sigma=\theta$, $\delta=0$, $f=t^n-\lambda$ and $f=gh$.\\
$(b)$ gives the $\theta$-constacyclicity condition for the code.\\
 $(c)$ we get the standard generating matrix of a 
 $\theta$-constacyclic code (cf. \cite{BSU}).
\item[(5)] If $A=\mathbb{F}_q$ is a finite field and $\theta\in Aut(\mathbb F_q)$
we get the skew codes defined in several papers.  Notice that, as
mentioned in example \ref{examples of Skew polynomial rings}(4) all the Ore
extensions over a finite field are of this form.
\item[(6)] Of course, over a finite ring we can also consider Ore extensions of derivation type.
For instance, let $R$ be the Ore extension
$R:=\mathbb F_p[x]/(x^p-1)[t;\frac{d}{dx}]$, where
$\frac{d}{dx}$ denotes the usual derivation.   $f(t)=t^p-1$ is in
fact a central polynomial in $R$.
Although this polynomial is the standard one for
building cyclic codes we will see many differences in the case of cyclic
($id.,\delta$)-codes.  First let us give the form
of the (right) roots of $t^p-1$ in $A:=\mathbb F_p[x]/(x^p-1)$.   We must
find the elements $q(x)\in A$ such that $N_p(q(x))=1$.   It is easy to compute
that $N_p(q(x))=q(x)^p+\frac{d^{p-1}}{dx}(q(x))$ (or cf \cite{L2}).
Hence since $x^p=1$, we have
$N_p(q(x))=q(x)+\frac{d^{p-1}}{dx}(q(x))$.  Set
$q(x)=\sum_{i=0}^{p-1}a_ix^i$.  One can check that $N_p(q(x))=1$ if and only if
$\sum_{i=0}^{p-2}a_i=1$.  In order to be concrete, let us fix $p=5$.
In this case $x$ and $x+x^4$ are roots of $t^5-1$ and one can easily compute that
the polynomial $g(t):=t^2-2xt+x^2-1$ is in fact the least left common multiple of
$t-x$ and $t-(x+x^4)$ in $R$.  A simple reasoning involving the division
algorithm then shows that $g(t)$ is a right (and hence left, since $f(t)$ is
central) factor of $t^5-1$.  The generating matrix of the cyclic
($id.,\frac{d}{dx}$)-code corresponding to the left module
$Rg/Rf$ is given by:
$$
\begin{pmatrix}
x^2-1 & -2x & 1 & 0 & 0 \\
2x & x^2+2 & -2x & 1 & 0 \\
2 & 4x & x^2 & -2x & 1
\end{pmatrix}
$$
\end{enumerate}
\end{examples}

\bigskip

Property $(b)$ in the above theorem \ref{generic matrices} characterizes the codes that can be obtained using a factor of a monic polynomial $f$.

\begin{definition}
A monic polynomial $f(t)\in R=A[t;\sigma,\delta]$ is invariant if
$Rf(t)=f(t)R$.
\end{definition}
Let $C(t)=Rg(t)/Rf(t)$ be a module code, where $f(t),g(t)\in R$ are monic polynomials such that $Rf(t)\subseteq Rg(t)$.
Remark that if either $f(t)$ or $g(t)$ is invariant then we can write
$f(t)=h'(t)g(t)=g(t)h(t)$, for some monic polynomials $h(t),h'(t)\in R$.   When there exist monic polynomials $h(t),h'(t)\in R$ such that 
$f(t)=h'(t)g(t)=g(t)h(t)$ the cyclic module
$Rg(t)/Rf(t)$ can be described via annihilators and the code $C$ via control matrices.  We start with
the following easy lemma.  The proof is left to the reader.

\begin{lemma}
Let $f,g,h,h' \in R$ be monic polynomials such that $f=gh=h'g$.   Then
\begin{enumerate}
\item[(a)] $gR=ann_R(h'+fR)$ \, and  \,$gR/fR=\{p+fR \,|\, p\in ann_R(h'+fR)\}$.
\item[(b)] $Rg=ann_R(h+Rf)$ \, and \, $Rg/Rf=\{p+Rf \,|\, p\in ann_R(h+Rf)\}$.
\end{enumerate}
\end{lemma}


\begin{theorem}
\label{the sigma delta code as annihilator}
Let $f,g,h,h'\in R$ be monic polynomials such that $f=gh=h'g$ and let $C$ denote the code corresponding
to the cyclic module $Rg/Rf$.  Then the following statements are equivalent:
\begin{enumerate}
\item[(i)] $(c_0,\dots,c_{n-1})\in C$,
\item[(ii)] $(\sum_{i=0}^{n-1}c_it^i)h(t)\in Rf$,
\item[(iii)] $\sum_{i=0}^{n-1}c_iT_f^i(\underline{h})=\underline{0}$,
\item[(iv)] $\sum_{j=0}^{n-1}(\sum_{i=j}^{n-1}c_if_j^i(\underline{h}))N_j(C_f)=\underline{0}$.
\end{enumerate}
\end{theorem}
\begin{proof}
\noindent (i) $\Leftrightarrow$ (ii) This is just the definition of $ann_{R}(h+Rf)$.

\noindent (ii) $\Leftrightarrow$ (iii) This comes from Lemma \ref{f divise à droite pq}.

\noindent (iii) $\Leftrightarrow$ (iv) It was mentioned in \ref{examples of PLT}  (4) that, for $\alpha\in A$ and $v\in V$, we have 
$T^n(\alpha v)=\sum_{i=0}^nf^n_i(\alpha)T^i(v)$.  Similarly we have,
for any $i\ge 0$
$T_f^i(v)=\sum f^i_j(v)N_j(C_f)$.  This formula was proved in \cite{L}.
\end{proof}

In view of the above it seems natural to set the following definition.

\begin{definition}
\label{definition of Control matrix}
For a left (resp. right) linear code $C\subseteq A^n$, we say that a matrix $H$ is a control matrix if
 $C=lann (H)$ (resp. $C=rann (H)$).
\end{definition}

From the above theorem \ref{the sigma delta code as annihilator}(iii) we immediately get the following corollary.

\begin{corollary}
\label{a control matrix given by T_f}
For a code $C$ determined by the left $R$-module $Rg/Rf$ such that there exist
monic polynomials $h,h'\in R$ with $f=gh=h'g$ the matrix $H$ whose $i^{th}$ row is $T_f^{i-1}(\underline{h})$, for
$1\le i \le \deg (f) $ is a
control matrix.
\end{corollary}

We show hereunder that the above Theorem \ref{the sigma delta code as annihilator}  and
Corollary \ref{a control matrix given by T_f} give back the
control matrix of classical cyclic and skew cyclic codes.

\begin{examples}
\begin{enumerate}
\item[(1)] Let $f(t)=t^n-1\in R=F[t]$, where $F$ is a (finite) field.
and let $g(t),h(t)\in R$ be such that $t^n-1=g(t)h(t)=h(t)g(t)$.
We write $h(t)=\sum_{i=0}^k h_it^i$.  For
$\underline{v}=(v_0,\dots,v_{n-1})\in k^n$, the action of $T_f^i$ is given by
$T_f^i(\underline{v})=(v_0,\dots,v_{n-1})C^i$, where $C$ is the companion matrix
associated to the polynomial $t^n-1$.  Theorem \ref{the sigma delta code as annihilator}
shows that a control matrix associated to the code $C$ corresponding to
$Rg/Rf$ is the following:
$$
\begin{pmatrix}
h_0 & h_1 & \dots & h_k & 0 & 0 & 0 \\
0 & h_0 & \dots & h_{k-1} & h_k & 0 & 0 \\
\vdots & \dots &  & \dots & \dots & \dots & 0 \\
h_k & 0 & \dots & h_0 & \dots & \dots & h_{k-1} \\
\vdots & \vdots & • & 0 &  & • & • \\
 & • & • & \vdots & • & h_0 & • \\
• & • & • & 0 & • & • & h_0
\end{pmatrix}
$$
Of course, the dimension of $Rg/Rf$ is equal to $k$ and hence the
rank of the control matrix must be $n-k$.  In other words, any set
of $n-k$ independent columns of the above matrix will have $C$ as
its (left) kernel. Since the last $n-k$ columns are in echelon form,
they are independent and hence these last columns give as well a
control matrix, say $H$.  Since $F$ is commutative we can see the
code as a right linear code and use the standard transposition to
get the control matrix of this "right" linear code.  A control
matrix for $C$ considered as a right linear code is thus just the
transpose of $H$.  This is the standard control matrix.
\item[(2)] In the same way as (1) we can consider the $\theta$-cyclic codes and obtain their
control matrices retrieving formulas proved elsewhere (e.g.
\cite{BGU}). In fact we more generally consider the following
situation Let $A$ be a ring and $\sigma$ an automorphism of $A$
(classically $A$ is a finite field and $\sigma$ is the Frobenius
automorphism). Assume $t^n-1=gh=h'g$, where $g,h,h'\in R$ are monic
polynomials. Let us write $h(t)=\sum_{i=0}^k h_it^i$, with $h_k=1$.
The pseudo-linear transformation defined by $f(t)=t^n-1$ is the map
$T_f: A^n\longrightarrow A^n$ defined by
$T_f(\underline{v})=\sigma(\underline{v})C$, where $C$ is the
companion matrix associated to $t^n-1$ and $\underline{v}\in A^n$.
It is easy to check that the following matrix $H$ is a control
matrix for the code $C$ determined by the module $Rg/Rf$:

$$
H=\begin{pmatrix}
h_0 & h_1 & \dots & h_k & 0 & 0 & 0 \\
0 & \sigma(h_0) & \dots & \sigma(h_{k-1}) & \sigma(h_k) & 0 & 0 \\
0 & 0 & \sigma^2(h_0) & \dots & \dots &  \dots   & \dots  \\
\vdots & \dots &  & \dots & \dots & \dots & 0 \\
0 & 0 & \dots & \dots& \dots & \dots & \sigma^{n-k}(h_{k}) \\
\vdots & \vdots & • & 0 &  & • & • \\
 & • & • & \vdots & • & h_0 & • \\
• & • & • & 0 & • & • & \dots
\end{pmatrix}
$$
So the last $n-k$ columns are in echelon form and hence linearly independent.  The dimension of the code being
equal to $k$, in good cases (e.g. if the ring is a field),  this means that they define
a control matrix as well.  The transpose of these last columns is exactly the control matrix obtained
by other authors in the case when $A$ is a commutative field.

\item[(3)] Let us now give an example of a cyclic code using a derivation.
Let $A$ be a ring and $\delta$ be a (usual) derivation on $A$.  For $a\in A$ we consider
the polynomial $f(t):=(t^2-a)^2\in A[t;\delta]$ and put $g=h=t^2-a$.
We easily compute $f(t)=t^4-2at^2-2\delta(a)t-\delta^2(a)+a^2$.
The generic matrix $G$ and control matrix $H$ are equal:
$$
H= \begin{pmatrix}
-a & 0 & 1 & 0 \\
-\delta(a) & -a & 0 & 1 \\
-a^2 & 0 & a & 0 \\
a\delta(a)-\delta(a)a & -a^2 & \delta(a) & a
\end{pmatrix}
$$
One can check that $\underline{g}H=(-a,0,1,0)H=(0,0,0,0)$.  Set $H_1,H_2,H_3,H_4$ to represent the different
columns of $H$,  then $H_1+H_3(-a)+H_4\delta(a)=0\in A^4$ and $H_2+aH_4=0\in A^4$.  Let $H'$ be the $4\times 2$
matrix $H'=(H_3,H_4)$.  We easily get that $lann (H')=lann (H)=C$.  This shows that $H'$ is a control matrix of the code $C$.
\item[(4)] We now compute a control matrix of the  cyclic code given in the above example \ref{examples of generic
matrices} (4).  We have $R:=\mathbb F_5[x]/(x^5-1)[t;\frac{d}{dx}]$, and $f(t)=t^5-1$.
This last polynomial is central and can be factorized as
$f(t)=g(t)h(t)=h(t)g(t)$ where $g(t):=t^2-2xt+x^2-1$ and
$h(t)= t^3 + 2xt^2 + (3x^2+2)t + (4x^3+3x)$.
The code we are considering corresponds to the module $Rg(t)/(t^5-1)$.  The
control matrix is given by the matrix $H\in M_5(\mathbb F_5)$ whose rows are given by
$T_f^i(\underline{h})$, $0\le i \le 4$.  The first row is thus $\underline{h}$
the second row is $\underline{h}C_f + \frac{d}{dx}(\underline{h})$.  Here $C_f$ is the
companion matrix of $t^5-1$ and acts as cyclic permutation.   Hence we get
$$
H=\begin{pmatrix}
4x^3+3x & 3x^2+2 & 2x & 1 & 0 \\
2x^2+3 & 4x^3+4 & 3x^2+4 & 2x & 1 \\
4x+1 & 4x^2+2 & 4x^3 & 3x^2+1 & 2x \\
2x+4 & 2x+1 & x^2+2 & 4x^3+6x & 3x^2+3 \\
3x^2 & 2x+1 & 4x+1 & 3x^2+3 & 4x^3+2x
\end{pmatrix}
$$
\end{enumerate}
\end{examples}

\section{($\si,\de$)-$W$-codes}

\vspace{4mm}

We will consider cyclic ($f(t),\sigma,\delta$)-codes corresponding to left cyclic 
modules of the form $Rg(t)/Rf(t)$ where $f(t),g(t)\in R=A[t;\sigma,\delta]$ are monic polynomials but $g(t)$ is a Wedderburn polynomial as explained in the following definitions.
\begin{definitions}
\begin{enumerate}
\item[(a)] A monic polynomial $g(t)\in R=\Ore$ of degree $r$ is a Wedderburn
polynomial if there exist elements $a_1,\dots,a_r\in A$ such that
$Rg(t)=\bigcap_{i=0}^rR(t-a_i)$.  We will refer to these polynomials as $W$-
polynomials.
\item[(b)]  The $n\times r$
generalized Vandermonde matrix defined by $a_1,\dots,a_r$ is given by:
$$
V_n(a_1,\dots,a_r)=
    \begin{pmatrix}
1 & 1 & \dots & 1 \\
a_1 & a_2 & \dots & a_r \\
\dots & \dots & \dots & \dots \\
N_{n-1}(a_1) & N_{n-1}(a_2) & \dots & N_{n-1}(a_r)
\end{pmatrix}.
$$
Recall that, for $0\le i \le n-1,\; N_i(a)$ is the evaluation of $t^i$ at $a\in A$ (cf. Definitions \ref{definitions of polynomial maps and pseudo linear maps}).
\item[(c)] A ($\si,\de$)-$W$-code $C\subseteq A^n$ is the set of $n$-tuples in 
$A^n$ corresponding to a cyclic left $R$-module of the form 
$Rg(t)/Rf(t)$ such that $g(t)$ is a $W$-polynomial.
\end{enumerate}
\end{definitions}
Wedderburn polynomials have been studied in details in \cite{LL} and \cite{LLO}.
 The generic matrix corresponding to a ($\si,\de$)-$W$-code is the standard one 
described in \ref{generic matrices}.  But an easy control matrix can be obtained as described in the 
next proposition.

Let us first state a general lemma.

\begin{lemma}
Let $f(t),g(t),h(t)\in R=\Ore$ be monic polynomials such that $f(t)=h(t)g(t)$.
Then $Rg(t)/Rf(t)=\{p(t)g(t)+Rf(t)\,|\,\deg p(t)<\deg h(t)\}$. 
\end{lemma}
\begin{proof}
This is obvious: if $m(t)g(t)+Rf(t) \in Rg(t)/Rf(t)$, dividing by the monic 
polynomial $h(t)$, we can write $m(t)=q(t)h(t)+p(t)$ with $\deg p(t)<\deg h(t)$
and we have $m(t)g(t)+Rf(t)=p(t)g(t)+Rf(t)$.
\end{proof}
\vspace{4mm}

\begin{proposition}
\label{generating and control matrix for a W-code}
Let $f(t),g(t)\in R=\Ore$ be monic polynomials of degree $n$ and $r$ 
respectively.  Suppose that $g(t)$ is a Wedderburn polynomial with $f(t)\in 
Rg(t)$ and let $C$ be 
the ($\sigma, \delta$)-$W$-code of length $n$ corresponding to the left
cyclic $R$-module $Rg(t)/Rf(t)$.  Let $a_1,\dots,a_r\in A$ be such that 
$Rg(t)=\bigcap_{i=0}^rR(t-a_i)$.
Then $(c_0,c_1,\dots,c_{n-1}) \in C$
if and only if $(c_0,c_1,\dots,c_{n-1})
V_{n}(a_1,\dots,a_r)=(0,\dots,0)$.
\end{proposition}
\begin{proof}

\noindent Let us remark that a polynomial $h(t)=\sum_{i=0}^{n-1}h_it^i\in Rg(t)$ if and only if
 $h(a_i)=0$ for all $1\le i \le r$.  Since $(h(a_0),\dots,h(a_r))=(h_0,\dots, h_{n-1})V_n(a_1,\dots,a_r)$,
 we have $h(t)\in Rg(t)$ if and only if $(h_0,\dots, h_{n-1})V_n(a_1,\dots,a_r)=(0,\dots,0)$.
 This yields the thesis.
\end{proof}

The proposition above amounts to say  that a control matrix is given by the Vandermonde matrix $V_n(a_1,\dots,a_r)$.
The Vandermonde matrix determined by Wedderburn
polynomial $g(t)$ can thus  be used as a control matrix for                                                                                                                                                                                                                                                                                                                                                                                                                                                                                                                                                                                                                                                                                                                                                                                                                                                                                                                                                         the ($\sigma,\delta$)-W-code $C$.

\begin{remarks}
\begin{enumerate}
\item[(1)] The Vandermonde matrices  are strongly related to Wronskian matrices
and to Noncommutative symmetric functions.  In an ($\sigma,\delta)$-setting
information can be found in \cite{DL}.  In particular, in this reference an 
axiomatic method is developed in order to compute
the least left common multiple of polynomials of the form $t-a_1,\dots,t-a_n$.
\item[(2)] In general the existence of a least left common multiple of linear
polynomials of the form $t-a_1,\dots, t-a_r$ is not guaranteed (for a general 
ring $A$).  The exact necessary and sufficient conditions 
for the existence of a LLCM of such
polynomials is given in Theorem 7.2 in \cite{DL}.
\item[(3)] If $A$ is a division ring the existence of LLCM of $t-a_1,\dots, t-a_r$ is clear but its degree can be less then $r$ even if the elements $a_1,\dots,a_r$ are all distinct.  For several
necessary and sufficient conditions for this the degree to be equal to $r$ we
refer the reader to \cite{LL} and \cite{LLO}.   In the case of a finite field
$\mathbb F_q$ such a condition can be find in \cite{MS}, pp 117-119.
\end{enumerate}
\end{remarks}

   The next theorem gives a characterization of the W-polynomials in  $R=\mathbb F_q[t;\theta]$.

\begin{theorem}
Let $p$ be a prime number and $n\in \mathbb N$.  Let also $R$ be the Ore
extension $R=\mathbb F_q[t;\theta]$, where $q=p^n$ and $\theta$ is the Frobenius
map.  We extend $\theta$ to $R$ by defining $\theta(t)=t$. Then:
\begin{enumerate}
\item[(a)] The polynomial $G(t)=t^{(p-1)n+1}-t$ (resp. $G_0(t)=t^{(p-1)n}-1$ ) is the least left common multiple of all the linear polynomials $t-a$,
$a\in \mathbb F_q$ (resp. $0\ne a\in \mathbb F_q$).
\item[(b)] Let $G(t)$ and $G_0(t)$ be as in the statement (a) above.  For any $h(t)\in \mathbb F_q[t;\theta]$, we have $G(t)h(t)=\theta(h(t))G(t)$.  The polynomial $G_0(t)=t^{(p-1)n}-1$ belongs to the center of $R$.
\item[(c)] Let $G(t)$ and $G_0(t)$ be as in the statement (a) above.  If $g(t),h(t)\in \mathbb F_q[t;\theta]$ are monic polynomials
such that $h(t)g(t)=G(t)$, then $\theta(g(t))h(t)=G(t)$.  Similarly if
$h(t)g(t)=G_0(t)$ then $g(t)h(t)=G_0(t)$.
\item[(d)] The W-polynomials are exactly the right (and left) factors of the
polynomial $G(t)$ mentioned in statement (a).
\end{enumerate}
\end{theorem}
\begin{proof}
\noindent (a) The fact that the polynomial $G(t)$ is a least left common multiple
of the polynomials $t-a$, such that $a\in \mathbb F_q$ was proved in Theorem 2.3
in \cite{L2}. 

\noindent (b) Since $\theta^n=id.$, it is easy to check that
$G(t)a=\theta(a)G(t)$ and $G_0(t)a=aG_0(t)$, for any $a\in A$.  This yields the
results.

\noindent (c) Multiplying the equality $h(t)g(t)=G(t)$ by $g(t)$ on the right
we get $h(t)g(t)^2=G(t)g(t)=\theta(g(t))G(t)=\theta(g(t))h(t)g(t)$.  Since $R$ is
an integral domain we obtain $G(t)=h(t)g(t)=\theta(g(t))h(t)$.  The statement related
to $G_0(t)$ is obtained similarly.

\noindent (d) Let $g(t)$ be a Wedderburn polynomial, say
$Rg(t)=\bigcap_{i=0}^rR(t-a_i)$.  Since $G(a_i)=0$, for any $i=0,\dots r$, we
immediately get that $G(t)\in Rg(t)$.  This shows that $g(t)$ is a right factor
of $G(t)$.  By its definition, $G(t)$ is a Wedderburn
polynomial.  It is a standard fact that factors of Wedderburn polynomials are
themselves Wedderburn (cf.  \cite{LL}).
\end{proof}

This theorem also shows that even without knowing the roots of the Wedderburn polynomial $g(t)$,  we immediately get a control matrix.  This is the content of the following corollary.

\begin{corollary}
Let $g(t)\in R=\mathbb{F}_q[t;\theta]$ be a $W$-polynomial of degree $r$.
As in the previous theorem let us denote $G_0(t)=t^{(p-1)n}-1$ and 
$G(t)=t^{(p-1)n+1}-t$.
Let $g(t),h(t)\in R$ be monic polynomials such that $G(t)=h(t)g(t)$ and consider the cyclic ($G(t),\theta, 0$)-code $C$ defined by the $R$-module
$Rg(t)/RG(t)$.
\begin{enumerate}
\item[(a)] There exists $1\le l\le n$ such that $\theta^l(g(t))=g(t)$ and 
we then have $G(t)=h(t)g(t)=g(t)\theta^{l-1}(h(t))$.
\item[(b)] The control matrix of the code $C$ is given by the matrix whose rows are $T_G^i(\theta^{l-1}(\underline{h}))$ for $0\le i \le (p-1)n$.
\item[(c)] Suppose the polynomial $g(t)$ is such that $g(0)\ne 0$. Then there 
exists $h'(t)\in R$ such that $G_0(t)=h'(t)g(t)=h'(t)g(t)$.
  The control matrix of the code corresponding to the cyclic module
$Rg(t)/RG_0(t)$ is given by the matrix whose rows are $T_{G_0}^i(\underline{h'})$ for $0\le i \le (p-1)n-1$.
\end{enumerate}
\end{corollary}
\begin{proof}
The proofs are left to the reader.
\end{proof}

\vspace{4mm}

\centerline{\large {\bf Acknowledgments} }

\vspace{3mm}

This paper was partially prepared while the first author visited the 
University of Artois.  He would like to thank the members of this institution 
for their kind hospitality.  The first author also acknowledges the support from King Khalid University 
of Saudi Arabia (program for research and researchers number $KKU S179 33$).


\vspace{4mm}

\begin{thebibliography}{xx}

\bibitem{A}{S.A. Amitsur}, {\em Derivations in simple rings},
Proc. London Math. Soc. (1957) s3-7 (1) 87-112.

\bibitem{BGU}{D. Boucher, W. Geiselmann, and F. Ulmer}, {\em Skew-cyclic codes},
Applicable Algebra in Engineering, Communication
and computing, ({\bf 18}) (4), (2007), 379-389.

\bibitem{BU}{D. Boucher, F. Ulmer}, {\em Linear codes using skew polynomials with 
automorphisms and derivations}, to appear in Design, Codes and Cryptography. 
manuscript

\bibitem{BSU}{D. Boucher, P. Sol\'e and F. Ulmer}, {\em Skew constacyclic codes over Galois rings, Advances in Mathematics of communications}, ({\bf 2}), (2008),
 273-292.

\bibitem{DL}{J. Delenclos and A. Leroy},
{\em Noncommutative symmetric functions and W-polynomials}
Journal of Algebra and its Applications, ({\bf 6}) (5), (2007), 815-837.

\bibitem{G}{M. Giesbrecht}, {\em Factoring in skew polynomial rings over finite 
fields}, Journal of Symbolic 
computations, ({\bf 26}) (4), (1998) 463-468.

\bibitem{J}{N. Jacobson}, {\em On pseudo linear transformations}, Annals of Mathematics ({\bf 38}), (1937) 484-507.

\bibitem{JN}{S.K. Jain, S.R. Nagpaul}, {\em Topics in Applied Abstract Algebra},
The Brooks/Cole Series in Advanced Mathematics.

\bibitem{LL}{ T.Y. Lam and A. Leroy },
{\em Wedderburn polynomials over division rings, I,}, Journal of Pure
and Applied Algebra, {\bf 186}, (2004), 43-76.

\bibitem{LLO}{ T.Y. Lam,  A. Leroy and A. Ozturk},
{\em Wedderburn polynomial over division rings, II,} Proceedeing of
conferences held in Chennai at the Ramanujan Institue (Indes)
Contemporary mathematics {\bf 456}, (2008), 73-98.

\bibitem{L}{ A. Leroy},
{\em Pseudo-linear transformation and evaluation in Ore extension},
Bull. Belg. Math. Soc. {\bf 2} (1995), 321-345.

\bibitem{L2}{ A. Leroy}, {\em Noncommutative polynomial maps},
Journal of Algebra and its Applications (2012).

\bibitem{LS}{S.R. L\'opez-Permouth and S. Szabo}, {\em Convolutional codes 
 with additional algebraic structures}, Journal of Pure and Applied Algebra (2012) doi:10.1016/J.paa.2012.09.017.

\bibitem{MS}{ F.J. MacWilliams and N.J.A. Sloane}, {\em The theory of error correcting codes}, North-Holland, Amsterdam/New York/Oxford, 1978.

\bibitem{S}{P.Sol\'e}, {\em Codes over rings}, Proceeding of the CIMPA Summer School Ankara, Turkey, 18-29 August 2008.

\bibitem{SY}{P.Sol\'e and O. Yemen}, {\em Binary quasi-cyclic codes of index 2 and skew polynomial rings}, to appear in finite fields and their applications, 
(2012).

\bibitem{W}{J. Wood}, {\em Code equivalence characterizes finite Frobenius rings},  Proc. Amer. Math. Soc. {\bf 136} (2008), 699-706.
\end{thebibliography}
\end{document}